\documentclass{amsart}
\usepackage{enumerate, amssymb, amsmath, rotating}%mathrsfs,verbatim,graphicx, , rotating, 
\usepackage{color}

\def\dist{\operatorname{dist}}

\def\N{\mathbb N}

\def\M{\mathcal M}
\def\eps{\varepsilon}
\def\span{\operatorname{span}}

\def\hd{\hat{\operatorname{d}}}
\def\wk{\operatorname{wk}}
\def\wck{\operatorname{wck}}
\def\ca{\operatorname{ca}}
\def\uc{\operatorname{uc}}
\def\fc{\operatorname{fix}_{c_0}}
\def\bc{\operatorname{bc}}
\def\cc{\operatorname{cc}}
\def\wcc{\operatorname{wcc}}
\def\Rcc{\operatorname{Rcc}}

\def\q#1{(#1$_q$)}
\def\qo#1{(#1$_q$)$_\omega$}
\def\qos#1{(#1$_q$)$_\omega^*$}
\def\qO#1{(#1$_q^\omega$)}
\def\qoO#1{(#1$_q^\omega$)$_\omega$}
\def\qosO#1{(#1$_q^\omega$)$_\omega^*$}

\def\clust{\operatorname{clust}}

\def\trn#1{{\left\vert\kern-0.3ex\left\vert\kern-0.3ex\left\vert\hskip 0,1 mm #1\hskip 0,1 mm\right\vert\kern-0.3ex\right\vert\kern-0.3ex\right\vert}}

\newtheorem{theorem}{Theorem}[section]
\newtheorem{proposition}[theorem]{Proposition}

\newtheorem{corollary}[theorem]{Corollary}
\newtheorem{question}[theorem]{Question}
\theoremstyle{definition}
\newtheorem*{remark}{Remark}
\newtheorem*{definition}{Definition}

\begin{document}

\title{Quantification of Pe\l czy\' nski's property (V)}
\author{Hana Kruli\v sov\'a}

\address{Department of Mathematical Analysis \\
Faculty of Mathematics and Physics\\ Charles University in Prague\\
Sokolovsk\'{a} 83, 186 \ 75\\Praha 8, Czech Republic}

\email{krulisova@karlin.mff.cuni.cz}


\thanks{The research was supported by the Grant No. 142213/B-MAT/MFF of the Grant Agency of the Charles University in Prague
and by the Research grant GA\v{C}R P201/12/0290.}

\begin{abstract}
A Banach space $X$ has Pe\l czy\' nski's property (V) if for every Banach space $Y$ every unconditionally converging operator $T\colon X\to Y$ is weakly compact. In 1962, Aleksander Pe\l czy\' nski showed that $C(K)$ spaces for a compact Hausdorff space $K$ enjoy the property (V), and some generalizations of this theorem have been proved since then. We introduce several possibilities of quantifying the property (V). We prove some characterizations of the introduced quantitative versions of this property, which allow us to prove a quantitative version of Pelczynski's result about $C(K)$ spaces and generalize it. Finally, we study the relationship of several properties of operators including weak compactness and unconditional convergence, and using the results obtained we establish a relation between quantitative versions of the property (V) and quantitative versions of other well known properties of Banach spaces.
\end{abstract}

\maketitle

%%%%%%%%%%%%%%%%%%%%%%%%%%%%%%%%%%%%%%%%%%%%%%%%%%%%%%%%%%%%%%%%%%%%%%%%%%%%%%%%
%%%%%%%%%%%%%%%%%%%%%%%%%%%%%%%%%%%%%%%%%%%%%%%%%%%%%%%%%%%%%%%%%%%%%%%%%%%%%%%%
\section{Introduction}
%%%%%%%%%%%%%%%%%%%%%%%%%%%%%%%%%%%%%%%%%%%%%%%%%%%%%%%%%%%%%%%%%%%%%%%%%%%%%%%%
%%%%%%%%%%%%%%%%%%%%%%%%%%%%%%%%%%%%%%%%%%%%%%%%%%%%%%%%%%%%%%%%%%%%%%%%%%%%%%%%
A Banach space $X$ is said to have \emph{Pe\l czy\' nski's property (V)} if for every Banach space $Y$ every unconditionally converging operator $T\colon X\to Y$ is weakly compact. Recall that a linear operator $T\colon X\to Y$ is \emph{weakly compact} if the image under $T$ of the unit ball of $X$ is a relatively weakly compact set in $Y$. We say that a bounded linear operator $T\colon X\to Y$ is \emph{unconditionally converging} if $\sum_{n} Tx_n$ is an unconditionally convergent series in $Y$ whenever $\sum_{n} x_n$ is a weakly unconditionally Cauchy series in $X$.

Spaces known to enjoy the property (V) are for example $C(K)$ for a compact Hausdorff space $K$; this result from 1962 is due to A. Pe\l czy\' nski \cite{pel}. Several generalizations of Pe\l czy\' nski's theorem have been proved since then. W. B. Johnson and M. Zippin have shown that all real $L^1$~preduals have the property (V) (see \cite{jz}). H. Pfitzner has proved that all $C^*$-algebras enjoy it as well (see \cite{pfitzner}).

The aim of this paper is to explore some possibilities of quantifying Pe\l czy\' nski's property (V). 
Our inspiration comes from plenty of recently published quantitative results. Let us mention for example quantitative versions of Krein's theorem \cite{fhmz, granero, ghm, cmr}, the Eberlein-\v Smulyan and the Gantmacher theorem \cite{ac}, James' compactness theorem \cite{cks, ghp}, weak sequential continuity and the Schur property \cite{kps, ks-schur}, the Dunford-Pettis \cite{kks} and the reciprocal Dunford-Pettis property \cite{qrdpp}, the Grothendieck property \cite{qgroth}, and the Banach-Saks property \cite{qbs}.

The main idea of quantifying an existing qualitative result is simple -- to replace an implication by an inequality. In case of the property (V) we will attempt to replace the implication
\begin{equation}
\label{impl}
T \text{ is unconditionally converging } \Rightarrow\ T \text{ is weakly compact}
\end{equation}
by an inequality
$$\begin{aligned} &\text{measure of weak non-compactness of } T \\
&\hskip 3 cm \leq C\cdot\text{measure of } T \text{ not being unconditionally converging},
\end{aligned}$$
where $C$ is some positive constant depending only on $X$.
These two measures should be positive numbers for each operator $T$ and should equal zero if and only if $T$ is weakly compact or unconditionally converging, respectively. This inequality then trivially includes the original implication, but it says even more.

In Section 2 we explain how to define the above mentioned measures and we introduce a quantitative version of the property (V). Section 3 is devoted to characterizations of a quantitative version of the property (V). Using these characterizations, in Section 4 we prove quantitative versions of the above-mentioned theorem of Pe\l czy\' nski and that of Johnson and Zippin. Section 5 describes a relationship of various properties of operators including weak compactness and unconditional convergence. These relationships are quantified, which enables us to establish a relation between a quantitative version of the property (V) and quantitative versions of some other well known properties of Banach spaces.

Throughout the paper, all Banach spaces can be considered either real or complex (most of the results are valid in both cases), unless stated otherwise. By an operator we always mean a bounded linear operator. If $X$ is a Banach space, we denote by $B_X$ its closed unit ball $\{x\in X:\,\|x\|\leq 1\}$ and by $U_X$ its open unit ball $\{x\in X:\,\|x\|<1\}$.
Every Banach space $X$ is considered canonically embedded into its bidual $X^{**}$.

%%%%%%%%%%%%%%%%%%%%%%%%%%%%%%%%%%%%%%%%%%%%%%%%%%%%%%%%%%%%%%%%%%%%%%%%%%%%%%%%
%%%%%%%%%%%%%%%%%%%%%%%%%%%%%%%%%%%%%%%%%%%%%%%%%%%%%%%%%%%%%%%%%%%%%%%%%%%%%%%%
\section{Quantification of Pe\l czy\' nski's property (V)}
%%%%%%%%%%%%%%%%%%%%%%%%%%%%%%%%%%%%%%%%%%%%%%%%%%%%%%%%%%%%%%%%%%%%%%%%%%%%%%%%
%%%%%%%%%%%%%%%%%%%%%%%%%%%%%%%%%%%%%%%%%%%%%%%%%%%%%%%%%%%%%%%%%%%%%%%%%%%%%%%%
In this section we remind the definition of the property (V). Then we define a few related quantities, which allow us to quantify the property (V). We first focus on a quantity which measures how far is an operator from being unconditionally converging. Then we remind some well known measures of weak non-compactness of sets and operators. Eventually, we introduce a quantitative version of the property (V).

%%%%%%%%%%%%%%%%%%%%%%%%%%%%%%%%%%%%%%%%%%%%%%%%%%%%%%%%%%%%%%%%%%%%%%%%%%%%%%%%
\subsection{Unconditionally converging operators and related quantities}
%%%%%%%%%%%%%%%%%%%%%%%%%%%%%%%%%%%%%%%%%%%%%%%%%%%%%%%%%%%%%%%%%%%%%%%%%%%%%%%%

\begin{definition}
A series $\sum_{n=1}^\infty x_n$ in a Banach space $X$ is
\begin{itemize}
\item \emph{unconditionally convergent}\/ if the series $\sum_{n=1}^\infty {t_n}x_n$ converges whenever $(t_n)$ is a bounded sequence of scalars,
\item \emph{weakly unconditionally Cauchy} (\emph{wuC} for short) if for all $x^*\in X^*$ the series $\sum_{n=1}^\infty |x^*(x_n)|$ converges.
\end{itemize}
\end{definition}

\begin{definition}
Let $X$, $Y$ be Banach spaces. An operator $T\colon X\to Y$ is \emph{unconditionally converging (uc)}\/ if for every weakly unconditionally Cauchy series $\sum_{n=1}^\infty x_n$ in $X$ the series $\sum_{n=1}^\infty Tx_n$ is unconditionally convergent.
\end{definition}

It is easy to see that an operator $T$ is unconditionally converging if and only if for every weakly unconditionally Cauchy series $\sum x_n$ in $X$ the series $\sum Tx_n$ is convergent. Indeed, the ``only if implication'' is trivial since every unconditionally convergent series is convergent. Suppose that $T$ sends wuC series to convergent series. If $\sum x_n$ is a wuC series in $X$ and $(t_n)$ is a bounded sequence of scalars, then $\sum t_n x_n$ is also wuC and hence $\sum t_n x_n$ converges. Therefore $T$ is uc.

Let $(x_n)$ be a bounded sequence in a Banach space $X$. Set
$$\ca\big((x_n)\big)=\inf_{n\in\N} \sup \{\|x_k-x_l\|:\ k,l\in\N,\ k,l\geq n\}.$$
This quantity is a measure of non-cauchyness of the sequence $(x_n)$. More precisely, $\ca\big((x_n)\big)$ is a positive number for every bounded sequence $(x_n)$ and it is equal to zero if and only if $(x_n)$ is Cauchy. Since we deal with Banach spaces only, the quantity $\ca$ measures non-convergence of sequences.

We are now prepared to define a quantity which measures how far is an operator $T$ from being unconditionally converging. Let $T\colon X \to Y$ be an operator between Banach spaces $X$ and $Y$. We set
$$\uc(T) = \sup\left\{\ca\left(\bigg(\sum_{i=1}^n T x_i\bigg)_n\right):\ (x_n)\subset X,\ \sup_{x^*\in B_{X^*}}\sum_{n=1}^\infty |x^*(x_n)|\leq 1\right\}.$$
Clearly, $\uc(T)=0$ provided $T$ is unconditionally converging. On the other hand, if $\sum x_n$ is a wuC series in $X$, then the sets
$$M_k=\bigg\{x^*\in X^*:\, \sum_{n=1}^\infty|x^*(x_n)| \leq k\bigg\},~k\in\N,$$
are closed, and $\bigcup_{k=1}^\infty M_k=X^*$. If we use Baire's theorem, it is not difficult to find a constant $C>0$ such that $\sum_{n=1}^\infty|x^*(x_n)| \leq C$ for all $x^*\in B_{X^*}$. From this we see that $\uc(T)=0$ if and only if $T$ is unconditionally converging.

%%%%%%%%%%%%%%%%%%%%%%%%%%%%%%%%%%%%%%%%%%%%%%%%%%%%%%%%%%%%%%%%%%%%%%%%%%%%%%%%
\subsection{Measuring non-compactness and weak non-compactness of sets and operators}
%%%%%%%%%%%%%%%%%%%%%%%%%%%%%%%%%%%%%%%%%%%%%%%%%%%%%%%%%%%%%%%%%%%%%%%%%%%%%%%%

We will use the following notation. For $A$, $B$ subsets of a Banach space $X$ we set
$$
\begin{aligned}
&\dist(a,B)=\inf\{\|a-b\|:\ a\in A,\ b\in B\}, \\
&\hd(A,B) = \sup\{\dist(a,B):\ a\in A\}.
\end{aligned}
$$
The former is the ordinary distance between the sets $A$ and $B$, the latter is the non-symetrized Hausdorff distance from $A$ to $B$.

Let $A$ be a bounded subset of a Banach space $X$. The Hausdorff measure of non-compactness of the set $A$ is defined by
$$
\begin{aligned}
\chi(A) &= \inf\{\hd(A,K): \emptyset\not=K\subset X \text{ is compact}\} \\
&= \inf\{\hd(A,F): \emptyset\not=F\subset X \text{ is finite}\}.
\end{aligned}
$$
It is easy to see that $\chi(A)=0$ if and only if the set $A$ is relatively compact.

There are many ways of measuring weak non-compactness.
The de Blasi measure of weak non-compactness of the set $A$, which is an analogue of the Hausdorff measure of non-compactness, is defined by
$$\omega(A)=\inf\{\hd(A,K): \emptyset\not=K\subset X \text{ is weakly compact}\}.$$
Clearly, $\omega(A)=0$ for any relatively weakly compact set $A$. De Blasi has proved (see \cite{deblasi}) that $\omega(A)=0$ if and only if $A$ is relatively weakly compact. For every bounded subset $A$ of a Banach space $X$ the inequality
\begin{equation}
\omega(A) \leq \chi(A)
\end{equation}
trivially holds.

Other most commonly used quantities measuring weak non-compactness are
$$
\begin{aligned}
\wk_X(A) &= \hd\big(\overline{A}^{w^*},X\big), \\
\wck_X(A) &= \sup\{\dist(\clust_{w^*}(x_n),X):\ (x_n) \text{ is a sequence in } A\}, \\
\gamma(A) &= \sup\{\|\lim_n\lim_m x^*_m(x_n) - \lim_m\lim_n x^*_m(x_n) \|:\ (x_n) \text{ is a sequence in } A,\\
&\hskip 35 mm (x^*_m) \text{ is a sequence in } B_{X^*}¨, \text{ and the limits exist}\}.
\end{aligned}
$$

Here $\overline{A}^{w^*}$ stands for the weak$^*$ closure of the set $A$ in the bidual space $X^{**}$ and $\clust_{w^*}(x_n)$ is the set of all weak$^*$ cluster points of the sequence $(x_n)$ in $X^{**}$. The quantity $\wck_X$ is related to the Eberlein-\v Smulyan theorem and the quantity $\gamma$ to the Grothendieck double limit criterion for weak compactness.

The above defined quantities are studied for example in \cite{ac} and the following relationships between them are proved there \cite[Theorem 2.3]{ac}. For every bounded subset $A$ of a Banach space $X$
\begin{equation}
\label{equiv}
\wck_X(A) \leq \wk_X(A) \leq \gamma(A) \leq 2\wck_X(A),
\end{equation}
\begin{equation}
\label{non-equiv}
\wk_X(A)\leq \omega(A).
\end{equation}
Moreover, all these quantities are measures of weak non-compactness in the sense that they are equal to zero if and only if the set $A$ is relatively weakly compact.
The estimates (\ref{equiv}) say that the measures $\wk_X$, $\wck_X$, and $\gamma$ are equivalent. The quantity $\omega$ is, however, not equivalent to the other three (see \cite[Corollary 3.4]{ac}), i.e. a Banach space $X$ exists such that there is no constant $C$ satisfying for every bounded $A\subset X$ the inequality $\omega(A)\leq C\wk_X(A)$.

An operator $T\colon X\to Y$ between Banach space $X$ and $Y$ is \emph{weakly compact} if the image $T(B_X)$ of the unit ball of $X$ under $T$ is relatively weakly compact. A natural way to measure how far is an operator $T\colon X\to Y$ from being weakly compact is to measure weak non-compactness of $T(B_X)$. We do it using the above defined measures of weak non-compactness of sets. Let us denote $\omega(T(B_X))$ simply by $\omega(T)$. Analogously $\gamma(T)$, $\wk_Y(T)$, and $\wck_Y(T)$ stand for $\gamma(T(B_X))$, $\wk_Y(T(B_X))$, and $\wck_Y(T(B_X))$, respectively.

The Gantmacher theorem states that an operator $T\colon X\to Y$ between Banach spaces $X$ and $Y$ is weakly compact if and only if the dual operator $T^*\colon Y^*\to X^*$ is weakly compact. This theorem has a quantitative version \cite[Theorem 3.1]{ac}. It says that for any operator $T$
\begin{equation}
\label{q-gantmacher}
\gamma(T)\leq\gamma(T^*)\leq2\gamma(T).
\end{equation}
The analogous result with the quantity $\omega$ in place of $\gamma$ does not hold (see \cite[Theorem 4]{at}). %Hnedy krokodyl rules.

%%%%%%%%%%%%%%%%%%%%%%%%%%%%%%%%%%%%%%%%%%%%%%%%%%%%%%%%%%%%%%%%%%%%%%%%%%%%%%%%
\subsection{Quantitative version of Pe\l czy\' nski's property (V)}
%%%%%%%%%%%%%%%%%%%%%%%%%%%%%%%%%%%%%%%%%%%%%%%%%%%%%%%%%%%%%%%%%%%%%%%%%%%%%%%%

\begin{definition}
Let $X$ be a Banach space. We say that $X$ has \emph{Pe\l czy\' nski's property (V)}\/ if for every Banach space $Y$ every unconditionally converging operator $T\colon X\to Y$ is weakly compact.
\end{definition}

The property (V) can be now quantified as follows.
\begin{definition}
We say that a Banach space $X$ has a quantitative version of Pe\l czy\' nski's property (V) -- let us denote it by \q{V} -- if there exists a~constant $C>0$ such that for every Banach space $Y$ and every operator $T\colon X\to Y$
\begin{equation}
\label{quantpel} \gamma(T)\leq C\cdot \uc(T).
\end{equation}
\end{definition}
If $X$ has a quantitative version of Pe\l czy\' nski's property \q{V}, then it also enjoy the original qualitative property (V). Indeed, for any uc operator $T$ we have $\uc(T)=0$, hence $\gamma(T)=0$ which means that $T$ is weakly compact.

One may ask what would happen if we use a different measure of weak non-compactness in (\ref{quantpel}). By replacing $\gamma$ with $\wk_X$ or $\wck_X$ we achieve nothing new since these quantities are equivalent. However, if we use $\omega$ instead of $\gamma$, we obtain a~stronger assertion. Proposition \ref{C(Omega)}\,(ii) shows that this quantification is really different.
\begin{definition}
We say that a Banach space $X$ has the property \qo{V} if there exists a~constant $C>0$ such that for every Banach space $Y$ and every operator $T\colon X\to Y$
$$\omega(T)\leq C\cdot \uc(T).$$
\end{definition}

There are other possibilities of quantifying the property (V). As we will see later, it sometimes seems to be more natural to quantify the inequality
$$T \text{ is uc } \Rightarrow\ T^* \text{ is weakly compact}$$
which is equivalent to (\ref{impl}) by Gantmacher's theorem.
\begin{definition}
We say that a Banach space $X$ has the property \qos{V} if there exists a~constant $C>0$ such that for every Banach space $Y$ and every operator $T\colon X\to Y$
$$\omega(T^*)\leq C\cdot \uc(T).$$
\end{definition}
Here we have no choice concerning the measure of weak non-compactness. If we used $\gamma(T^*)$ in place of $\omega(T^*)$, it would only yield a reformulation of the property \q{V} by the quantitative Gantmacher theorem (\ref{q-gantmacher}).

%%%%%%%%%%%%%%%%%%%%%%%%%%%%%%%%%%%%%%%%%%%%%%%%%%%%%%%%%%%%%%%%%%%%%%%%%%%%%%%%
%%%%%%%%%%%%%%%%%%%%%%%%%%%%%%%%%%%%%%%%%%%%%%%%%%%%%%%%%%%%%%%%%%%%%%%%%%%%%%%%
\section{Characterizations of a quantitative Pe\l czy\' nski's property (V)}
%%%%%%%%%%%%%%%%%%%%%%%%%%%%%%%%%%%%%%%%%%%%%%%%%%%%%%%%%%%%%%%%%%%%%%%%%%%%%%%%
%%%%%%%%%%%%%%%%%%%%%%%%%%%%%%%%%%%%%%%%%%%%%%%%%%%%%%%%%%%%%%%%%%%%%%%%%%%%%%%%
\label{charact}

Pe\l czy\' nski's property (V) has multiple different characterizations. It turns out that some of these characterizations can also be quantified. We will show that their quantitative versions are equivalent to a quantitative version of Pe\l czy\' nski's property (V).

%%%%%%%%%%%%%%%%%%%%%%%%%%%%%%%%%%%%%%%%%%%%%%%%%%%%%%%%%%%%%%%%%%%%%%%%%%%%%%%%
\subsection{Characterization through subsets of the dual space}
%%%%%%%%%%%%%%%%%%%%%%%%%%%%%%%%%%%%%%%%%%%%%%%%%%%%%%%%%%%%%%%%%%%%%%%%%%%%%%%%

\begin{proposition}
\label{pel-prop}
Let $X$ be a Banach space. The following assertions are equivalent.
\begin{enumerate}
\item $X$ has Pe\l czy\' nski's property (V).
\item Every $K\subset X^*$ which satisfies the condition $(*)$ below is weakly compact.
$$(*) \qquad \lim_{n\to\infty} \sup_{x^*\in K} |x^*(x_n)|=0 \text{ for every wuC series } \sum_{n=1}^\infty x_n \text{ in } X.$$
\end{enumerate}
\end{proposition}

This proposition, proven by Pe\l czy\' nski \cite[Proposition 1]{pel}, has its quantitative analogue. We have already explained in the previous section how to reformulate the former assertion quantitatively. We now define a quantity which is essential for quantifying the latter one, and then we prove that also quantitative versions of the assertions (1) and (2) are equivalent.

Let $X$ be a Banach space and $K$ be a bounded subset of $X^*$. We set
$$\eta(K)=\sup\Big\{\limsup_n \sup_{x^*\in K} |x^*(x_n)|:\,(x_n)\subset X,\,\sup_{x^*\in B_{X^*}}\sum_{n=1}^\infty |x^*(x_n)|\leq 1\Big\}.$$
This quantity measures to what extent $K$ fails to satisfy the condition $(*)$ from the Proposition \ref{pel-prop}\,(2). Obviously, $\eta(K)$ is positive for every bounded $K\subset X^*$ and equals zero if and only if $K$ satisfies the condition $(*)$.

%%%%%%%%%%%%%%%%%%%%%%%%%%%%%%%%%%%%%%%%%%%%%%%%%%%%%%%%%%%%%%%%%%%%%%%%%%%%%%%%

\begin{proposition}
\label{pel-prop-non*}
Let $X$ be a Banach space. The following assertions are equivalent.
\begin{enumerate}
\item[(1$_q$)]  $X$ has the property \q{V}, i.e. there exists $C>0$ such that for any Banach space $Y$ and any operator $T\colon X\to Y$ $$\gamma(T)\leq C\cdot \uc(T).$$
\item[(1$'_q$)] There exists $C>0$ such that for every operator $T\colon X\to \ell^\infty$ $$\gamma(T)\leq C\cdot \uc(T).$$
\item[(2$_q$)]  There exists $C>0$ such that for each bounded $K\subset X^*$ $$\gamma(K)\leq C\cdot \eta(K).$$
\end{enumerate}
\end{proposition}

This proposition follows immediately from the next one and the quantitative version of Gantmacher's theorem (\ref{q-gantmacher}). The preceding and the following proposition are much alike, in the latter one $\gamma(T)$ is replaced by $\gamma(T^*)$. Then the three assertions are equivalent ``with the same constant'' $C>0$. Thus the quantification of the property (V) of the form
$$\gamma(T^*)\leq C\cdot\uc(T)
$$
seems to be more natural here.

%%%%%%%%%%%%%%%%%%%%%%%%%%%%%%%%%%%%%%%%%%%%%%%%%%%%%%%%%%%%%%%%%%%%%%%%%%%%%%%%

\begin{proposition}
\label{pel-prop-*}
Let $X$ be a Banach space and $C>0$. The following assertions are equivalent.
\begin{enumerate}
\item[(1$_q$)$_C$] For any Banach space $Y$ and any operator $T\colon X\to Y$ $$\gamma(T^*)\leq C\cdot \uc(T).$$
\item[(1$'_q$)$_C$] For every operator $T\colon X\to \ell^\infty$ $$\gamma(T^*)\leq C\cdot \uc(T).$$
\item[(2$_q$)$_C$] For each bounded $K\subset X^*$ $$\gamma(K)\leq C\cdot \eta(K).$$
\end{enumerate}
\end{proposition}
\begin{proof}
We follow Pe\l czy\' nski's original proof \cite[Prop.\,1]{pel}, it only needs to be done more carefully.
The implication $(1_q)_C \Rightarrow (1'_q)_C$ is obvious.

Let us prove $(1'_q)_C \Rightarrow (2_q)_C$. Let $K$ be a bounded subset of $X^*$ and $\delta<\gamma(K)$. From the definition of $\gamma$ it is easily seen that a sequence $(x^*_n)$ in $K$ exists such that $\gamma(\{x^*_n:\,n\in\N\})>\delta$. Let us define $T\colon X\to \ell^\infty$ by $T(x)(n)=x^*_n(x)$, $n\in\N$, $x\in X$. For each $n\in\N$ set $p_n\big((a_k)\big)=a_n$, $(a_k)\in\ell^\infty$. Then $p_n\in(\ell^\infty)^*$, $\|p_n\|=1$. Moreover, $T^*p_n=x^*_n$, because for $x\in X$ we have $T^*p_n(x)=p_n(Tx)=x^*_n(x)$. Thus
$$\gamma(T^*)=\gamma(T(B_{(\ell^\infty)^*}) \geq \gamma(\{T^*p_n:\,n\in\N\}) = \gamma(\{x^*_n:\,n\in\N\})>\delta.$$
From $(1'_q)_C$ it follows that $\uc(T)>\frac{\delta}{C}$. By the definition of the quantity $\uc$ there is a wuC series $\sum x_n$ in $X$ with $\sup_{x^*\in B_{X^*}}\sum |x^*(x_n)|\leq 1$ such that
$\ca\left(\big(\sum_{i=1}^n Tx_i\big)_n\right)>\frac{\delta}{C}$.
The definition of $\ca$ gives indices $k_1<l_1<k_2<l_2<\dots$ such that for each $n\in\N$
\begin{equation}
\label{odhad1}
\begin{aligned}
\frac{\delta}{C} &< \left\|\sum_{i=k_n}^{l_n} Tx_i\right\|_{\ell^\infty} = \sup_{m\in\N}\left|\sum_{i=k_n}^{l_n} T(x_i)(m)\right| \\
&= \sup_{m\in\N}\left|\sum_{i=k_n}^{l_n} x^*_m(x_i)\right| \leq \sup_{x^*\in K}\left|\sum_{i=k_n}^{l_n} x^*(x_i)\right|. \\
\end{aligned}
\end{equation}
Let us define
$\widetilde{x}_n=\sum_{i=k_n}^{l_n}x_i$, $n\in\N$.
Then the series $\sum_n \widetilde{x}_n$ is wuC since $\sum_i x_i$ is wuC and for every $x^*\in X^*$
$$\sum_{n=1}^\infty|x^*(\widetilde{x}_n)|
= \sum_{n=1}^\infty\left|\sum_{i=k_n}^{l_n}x^*(x_i)\right|
\leq \sum_{n=1}^\infty\sum_{i=k_n}^{l_n}|x^*(x_i)|
\leq \sum_{i=1}^\infty |x^*(x_i)|.$$
Moreover,
$$\sup_{x^*\in B_{X^*}} \sum_{n=1}^\infty|x^*(\widetilde{x}_n)| \leq \sup_{x^*\in B_{X^*}}\sum_{i=1}^\infty |x^*(x_i)|\leq 1.$$
From (\ref{odhad1}) we have for each $n\in\N$
$$\sup_{x^*\in K}|x^*(\widetilde{x}_n)| = \sup_{x^*\in K}\left|\sum_{i=k_n}^{l_n} x^*(x_i)\right| > \frac{\delta}{C},$$
and so $$\limsup_{n\in\N}\sup_{x^*\in K} x^*(\widetilde{x}_n) > \frac{\delta}{C}.$$
Hence $\eta(K) > \frac{\delta}{C}$. As $\delta<\gamma(K)$ has been chosen arbitrarily, we obtain $\gamma(K)\leq C\cdot\eta(K)$.

It remains to prove the implication $(2_q)_C \Rightarrow (1_q)_C$. Let $Y$ be a Banach space and $T\colon X\to Y$ an operator. Let us fix $\delta<\gamma(T^*) = \gamma(T^*(B_{Y^*}))$. Set $K=\gamma(T^*(B_{Y^*}))$. Then $K$ is a bounded subset of $X^*$ and from $(2_q)_C$ we have $C\cdot\eta(K)\geq\gamma(K)>\delta$. By the definition of $\eta$ there is a wuC series $\sum x_n$ in X with $\sup_{x^*\in B_{X^*}} \sum |x^*(x_n)|\leq 1$ such that
$$
\begin{aligned}
\frac{\delta}{C} &< \limsup_n \sup_{x^*\in K} |x^*(x_n)| = \limsup_n \sup_{y^*\in B_{Y^*}} |T^*y^*(x_n)| \\
&= \limsup_n \sup_{y^*\in B_{Y^*}} |y^*(Tx_n)| = \limsup_n \|Tx_n\|.
\end{aligned}
$$
Thus $\ca\left(\big(\sum_{i=1}^n Tx_i\big)_n\right)>\frac{\delta}{C}$
and hence $\uc(T)\geq \frac{\delta}{C}$. Since $\delta<\gamma(T^*)$ is arbitrary, we conclude that $\gamma(T^*)\leq C\cdot \uc(T)$.
\end{proof}

%%%%%%%%%%%%%%%%%%%%%%%%%%%%%%%%%%%%%%%%%%%%%%%%%%%%%%%%%%%%%%%%%%%%%%%%%%%%%%%%

The following proposition provides an analogous characterization of the property \qos{V}.

\begin{proposition}
\label{pel-prop-*-omega}
Let $X$ be a Banach space and $C>0$. The following assertions are equivalent.
\begin{enumerate}
\item[(1$_q^{\,\omega}$)$_C$] For any Banach space $Y$ and any bounded linear operator $T\colon X\to Y$ $$\omega(T^*)\leq C\cdot \uc(T).$$
\item[(2$_q^{\,\omega}$)$_C$] For each bounded $K\subset X^*$ $$\omega(K)\leq C\cdot \eta(K).$$
\end{enumerate}
\end{proposition}

\begin{proof}
This proposition has the ``same'' proof as the previous one. The implication $(2_q^{\,\omega})_C \Rightarrow (1_q^{\,\omega})_C$ can be proven exactly the same way, we simply substitute $\omega$ for $\gamma$.

As for the converse implication, suppose that $(1_q^{\,\omega})_C$ holds, and let $K$ be a bounded subset of $X^*$ and $\delta<\omega(K)$. Let us define $T\colon X\to \ell^\infty(K)$ by $Tx(x^*)=x^*(x)$, $x^*\in K$, $x\in X$. For each $x^*\in K$ set $F_{x^*}(f)=f(x^*)$, $f\in\ell^\infty(K)$. Then $F_{x^*}\in(\ell^\infty(K))^*$, $\|F_{x^*}\|=1$, and $T^*F_{x^*}=x^*$, $x^*\in X^*$. Hence
$$\omega(T^*) = \omega(T^*(B_{(\ell^\infty(K))^*})\geq \omega(\{T^*F_{x^*}:\,x^*\in K\}) %= \omega(\{x^*:\,x^*\in K\})
=\omega(K)>\delta.$$
By $(1_q^{\,\omega})_C$, $\uc(T)>\frac{\delta}{C}$. We then continue just as in the proof of the implication $(1'_q)_C\Rightarrow(2_q)_C$ in the previous proposition to get $(2_q^{\,\omega})_C$.
\end{proof}

From the estimates (\ref{equiv}) and (\ref{non-equiv}) it follows that if some Banach space $X$ satisfies the condition $(1_q^{\,\omega})_C$ from the previous proposition \ref{pel-prop-*-omega}, then it also satisfies the condition $(1_q)_{2C}$ from Proposition \ref{pel-prop-*}.

Propositions \ref{pel-prop-non*}, \ref{pel-prop-*}, and \ref{pel-prop-*-omega} characterize only the properties \q{V} and \qos{V}. We do not have a similar characterization of the property \qo{V}.

%%%%%%%%%%%%%%%%%%%%%%%%%%%%%%%%%%%%%%%%%%%%%%%%%%%%%%%%%%%%%%%%%%%%%%%%%%%%%%%%
\subsection{Characterization of uc operators and its consequence}
%%%%%%%%%%%%%%%%%%%%%%%%%%%%%%%%%%%%%%%%%%%%%%%%%%%%%%%%%%%%%%%%%%%%%%%%%%%%%%%%

The following theorem is a well known characterization of unconditionally converging operators due to Pe\l czy\' nski (see e.g. \cite[p. 54, Exercise 8]{diestel}). It yields another characterization of the property (V). Since this result has its quantitative version (Theorem \ref{fixc0-quant} below), it gives another characterization of a quantitative version of the property (V).

\begin{theorem}
\label{char-uc}
Let $X$, $Y$ be Banach spaces and $T\colon X\to Y$ an operator. Then $T$ is unconditionally converging if and only if it does not fix any copy of $c_0$, i.e. there is no subspace $X_0\subset X$ isomorphic to $c_0$ such that $T\restriction_{X_0}$ is an isomorphism.
\end{theorem}

To quantify this proposition we will need the quantity $\fc$ which measures the failure of the condition that $T$ does not fix a copy of $c_0$. For a bounded linear operator $T\colon X\to Y$ we set
\[
\begin{split}
\fc(T) = \sup\Big\{
(\|U\|\|V\|)^{-1}:~&\exists X_0\subset X\text{ such that } T\restriction_{X_0} \text{ is an isomorphism} \\[-5 pt]
&\text{onto } T(X_0)\text{, and }~(T\restriction_{X_0})^{-1}=U\circ V \text{ for some}\\[-5 pt]
&\text{onto isomorphisms } U\colon c_0\to X_0,~V\colon T(X_0)\to c_0 \Big\}.
\end{split}
\]
If the set on the right is empty, we set $\fc(T)=0$. This happens if and only if $T$ does not fix a copy of $c_0$, for otherwise the set contains $(\|U\|\|V\|)^{-1}$, where $U\colon c_0\to X_0\subset X$ is an isomorphism onto $X_0$ such that $T\restriction_{X_0}$ is an isomorphism, and $V=(T\circ U)^{-1}$.

Let us explain why may this quantity be considered a measure of the failure of the condition that $T$ does not fix a copy of $c_0$.
First of all, note that $\fc(cT)=c\,\fc(T)$, $c>0$. This is important, for we need $\fc$ to be positively homogeneous like all the other quantities that we use.
Now, suppose that $T$ is an operator of norm $1$ which fixes a copy of $c_0$. Let $X_0$ be a subspace of $X$ isomorphic to $c_0$ such that $T\restriction_{X_0}$ is an isomorphism onto $T(X_0)\subset Y$. If we wanted to measure how ``nice'' is this isomorphism, we would have to take a closer look at $\|(T\restriction_{X_0})^{-1}\|$. If it equals $1$, then $T\restriction_{X_0}$ is an isometry. The greater is $\|(T\restriction_{X_0})^{-1}\|$, the more ``deforming'' is the isomorphism $T\restriction_{X_0}$. We thus see that $\|(T\restriction_{X_0})^{-1}\|^{-1}$ is a natural measure of ``niceness'' of $T\restriction_{X_0}$. In our case, we would like to measure how nice is the isomorphism $T\restriction_{X_0}$ and how nice copy of $c_0$ is $X_0$ in $X$ simultaneously. The operator $(T\restriction_{X_0})^{-1}$ factors through $c_0$ in a way that there are isomorphisms $U$ and $V$ like in the definition of $\fc$ such that $(T\restriction_{X_0})^{-1}=U\circ V$. So we replace $\|(T\restriction_{X_0})^{-1}\|$ by $\|U\| \|V\|$. The quantity $(\|U\|\|V\|)^{-1}$ not only measures ``niceness'' of $U\circ V$, but it also takes into account the isomorphism $U\colon c_0\to X_0$ itself. Eventually, the supremum over all suitable $X_0$, $U$ and $V$ is taken to measure how nicest an isomorphism on some nice copy of $c_0$ can we get.

The following theorem is a quantitative version of Theorem \ref{char-uc}. Both implications of the equivalence are replaced by inequalities between relevant measures.

\begin{theorem}
\label{fixc0-quant}
Let $X$ be a Banach space. For every Banach space $Y$ and every bounded linear operator $T\colon X\to Y$
$$\frac12\uc(T) \leq \fc(T) \leq \uc(T).$$
\end{theorem}
\begin{proof}
Let us start with the second inequality. If $\fc(T)=0$, it holds trivially. Suppose that $\fc(T)>0$, i.e. $T$ fixes a copy of $c_0$. Take $X_0$ a subspace of $X$ isomorphic to $c_0$ and $U\colon c_0\to X_0$, $V\colon T(X_0)\to c_0$ onto isomorphisms which satisfy $(T\restriction_{X_0})^{-1}=U\circ V$. Is it enough to show that $\uc(T) \geq (\|U\| \|V\|)^{-1}$.

For the series $\sum e_n$ in $c_0$ we have
$$\sup_{x^*\in B_{(c_0)^*}} \sum |x^*(e_n)| = \sup_{(a_n)\in B_{\ell^1}} \sum |a_n| = 1$$
and $\ca\left(\big(\sum_{i=1}^n e_i\big)_n \right) = 1$.
Set $f_n=\frac{1}{\|U\|} U e_n$, $n\in\N$. Then $\sum f_n$ is a wuC series in $X_0\subset X$, since $\sum e_n$ is wuC and $U$ is continuous. We have even
$$\sup_{x^*\in B_{X^*}} \sum |x^*(f_n)| = \sup_{x^*\in B_{X^*}} \sum \left|\left(\tfrac{1}{\|U\|} x^*\circ U\right)(e_n)\right| \leq \sup_{y^*\in B_{(c_0)^*}} \sum |y^*(e_n)| = 1.$$
Moreover,
$$
\begin{aligned}
\ca\left(\bigg(\sum_{i=1}^n T f_i\bigg)_n\right) &= \inf_{n\in\N} \sup_{k>l\geq n} \left\|\sum_{i=l+1}^k T\left(\tfrac{1}{\|U\|}Ue_i \right)\right\| \\
&= \frac{1}{\|U\| \|V\|} \inf_{n\in\N} \sup_{k>l\geq n} \|(T\circ U)^{-1}\| \left\|\sum_{i=l+1}^k (T\circ U)e_i\right\| \\
&\geq (\|U\| \|V\|)^{-1} \inf_{n\in\N} \sup_{k>l\geq n} \left\|\sum_{i=l+1}^k e_i\right\| \\
&= (\|U\| \|V\|)^{-1} \ca\left(\bigg(\sum_{i=1}^n e_i\bigg)_n\right) \\
&= (\|U\| \|V\|)^{-1}.
\end{aligned}
$$
It follows that $\uc(T) \geq (\|U\| \|V\|)^{-1}$, which is what we need.

We proceed to show the inequality $\uc(T)\leq 2\,\fc(T)$. It is trivial if $\uc(T)=0$. Suppose that $\uc(T)>0$ and fix $0<\delta<\uc(T)$. First we find $\eps>0$ satisfying $\uc(T)>\delta(1+\eps)$, and we set $\delta'=\delta(1+\eps)$. The definition of $\uc(T)$ gives a wuC series $\sum x_n$ in $X$ with $\sup_{x^*\in B_{X^*}} \sum |x^*(x_n)| \leq 1$ such that $\ca\left(\left(\sum_{i=1}^n T x_i\right)_n\right) > \delta'$. By the definition of the quantity $\ca$ we find indices $k_1<l_1<k_2<l_2<\dots$ \ such that $\left\|\sum_{i=k_n}^{l_n} Tx_i \right\| > \delta'$, $n\in\N$. Let us set $\widetilde{x}_n = \sum_{i=k_n}^{l_n} x_i$, $n\in\N$. Then $\sum \widetilde{x}_n$ is a wuC series in $X$ with
$$\sup_{x^*\in B_{X^*}} \sum_{n=1}^\infty |x^*(\widetilde{x}_n)| \leq \sup_{x^*\in B_{X^*}} \sum_{n=1}^\infty |x^*(x_n)| \leq 1.$$
For each $n\in\N$ we have $\|T \widetilde{x}_n\| > \delta'$, and so $\|\widetilde{x}_n\| > \frac{\delta'}{\|T\|}>0$. The series $\sum \widetilde x_n$ is wuC and therefore $\widetilde{x}_n\to 0$ weakly. By \cite[Proposition 1.5.4]{kalton}, there is a subsequence $(\widetilde{x}_{n_k})$ of $(\widetilde{x}_n)$ which is basic.
Since $T\widetilde{x}_{n_k}\to 0$ weakly by the continuity of $T$, and $\inf\{\|T\widetilde{x}_{n_k}\|:\,k\in\N\} \geq \delta'>0$, we can use theorem \cite[Proposition 1.5.4]{kalton} again to obtain a subsequence $(z_m)$ of $(\widetilde{x}_{n_k})$ such that $(T z_m)$ is a basic sequence in $Y$ with a basic constant $\bc(T z_m)< 1+\eps$.

Since $(z_n)$ is a basic sequence in $X$ for which $\sum z_n$ is wuC, and $\inf\{\|z_n\|:\,n\in\N\}>0$, $(z_n)$ is equivalent to the canonical basis of $c_0$ by \cite[Theorem 6.6]{morrison}. For the same reason the sequence $(T z_n)$ in $Y$ is also equivalent to the canonical basis of $c_0$. Hence both $\overline{\span}\{z_n:\,n\in\N\}$ and $\overline{\span}\{T z_n:\,n\in\N\}$ are isomorphic to $c_0$, and $T \restriction_{\overline{\span}\{z_n:\,n\in\N\}}$ is an isomorphism onto $\overline{\span}\{T z_n:\,n\in\N\}$.

Let us set $X_0=\overline{\span}\{z_n:\,n\in\N\}$ and define $U\colon c_0 \to X_0$ by $U(e_n)=z_n$, $n\in\N$. Then $U$ is an onto isomorphism. Further, set $V=(T\circ U)^{-1}$. We will prove that $(\|U\| \|V\|)^{-1} \geq \frac{\delta}{2}$.
For $(a_n)\in c_0$ we have
$$
\begin{aligned}
\left\|U\big((a_n)\big)\right\|&=\left\|\sum_{n=1}^\infty a_n z_n \right\| = \sup_{x^*\in B_{X^*}}\left|x^*\left(\sum_{n=1}^\infty a_n z_n\right)\right| \\
&\leq \sup_{x^*\in B_{X^*}} \sum_{n=1}^\infty |a_n||x^*(z_n)| \leq \sup_{n\in\N}|a_n| \sup_{x^*\in B_{X^*}} \sum_{n=1}^\infty |x^*(z_n)| \\
&\leq \|(a_n)\| \sup_{x^*\in B_{X^*}} \sum_{i=1}^\infty |x^*(\widetilde{x}_i)| \leq \|(a_n)\|,
\end{aligned}$$
and hence $\|U\| \leq 1$. If $(a_n)\in c_0$, we also have for each $n\in\N$
$$
\begin{aligned}
\delta'|a_n| &\leq |a_n| \|T z_n\| = \|a_n T z_n\| = \left\|\sum_{i=1}^n a_i T z_i - \sum_{i=1}^{n-1} a_i T z_i\right\| \\
&\leq \left\|\sum_{i=1}^n a_i T z_i\right\| + \left\|\sum_{i=1}^{n-1} a_i T z_i\right\|  \leq 2 \bc(T z_k) \left\|\sum_{k=1}^\infty a_k T z_k\right\| \\
&= 2 \bc(T z_k) \left\|\big(T\circ U\big)\big((a_k)\big)\right\|,
\end{aligned}
$$
which gives
$$
\begin{aligned}
\|(a_n)\| &= \sup_{n\in\N}|a_n| \leq \frac{2 \bc(T z_n)}{\delta'} \left\|\big(T\circ U\big)\big((a_n)\big)\right\| \\
&\leq \frac{2(1+\eps)}{\delta(1+\eps)}  \left\|\big(T\circ U\big)\big((a_n)\big)\right\| = \frac{2}{\delta} \left\|\big(T\circ U\big)\big((a_n)\big)\right\|.
\end{aligned}
$$
Hence $\|V\|=\left\|(T\circ U)^{-1}\right\| \leq \frac{2}{\delta}$, and we thus obtain $\left(\|U\|\|V\|\right)^{-1} \geq 1\cdot \frac{\delta}{2} = \frac{\delta}{2}$. Consequently, $\fc(T)\geq\frac{\delta}{2}$. This yields the desired inequality $\uc(T)\leq 2\,\fc(T)$.
\end{proof}

%%%%%%%%%%%%%%%%%%%%%%%%%%%%%%%%%%%%%%%%%%%%%%%%%%%%%%%%%%%%%%%%%%%%%%%%%%%%%%%%
%%%%%%%%%%%%%%%%%%%%%%%%%%%%%%%%%%%%%%%%%%%%%%%%%%%%%%%%%%%%%%%%%%%%%%%%%%%%%%%%
\section{Quantitative version of Pe\l czy\' nski's theorem and its generalizations}
%%%%%%%%%%%%%%%%%%%%%%%%%%%%%%%%%%%%%%%%%%%%%%%%%%%%%%%%%%%%%%%%%%%%%%%%%%%%%%%%
%%%%%%%%%%%%%%%%%%%%%%%%%%%%%%%%%%%%%%%%%%%%%%%%%%%%%%%%%%%%%%%%%%%%%%%%%%%%%%%%

A theorem of A. Pe\l czy\' nski from 1962 asserts that the space $C(K)$ of continuous real functions on a compact Hausdorff space $K$ has the property (V) (see \cite[Theorem 1]{pel}). Using a characterization of a quantitative Pe\l czy\' nski's property (V) from the section \ref{charact} we prove a quantitative strengthening of this theorem. The proof is inspired by Pe\l czy\' nski's original proof and it uses some results of \cite{qrdpp}.

\begin{theorem}
\label{q-pel}
Let $\Omega$ be a locally compact space. Then the space $C_0(\Omega)$ enjoys the quantitative property \qos{V} (and hence also \q{V}). More precisely, for every Banach space $Y$ and every operator $T\colon C_0(\Omega)\to Y$
$$\omega(T^*)\leq \pi \uc(T).$$
In the real case (i.e. if $C_0(\Omega)$ are real functions) the constant $\pi$ can be replaced by $2$.
\end{theorem}

\begin{remark}
It might seem that the quantification with $\omega$ in this theorem is stronger than the quantification through the inequality $\gamma(T^*)\leq C\, \uc(T)$ (which is equivalent to \q{V}), but it is not. In fact, by \cite[Theorem 7.5]{kks} the quantities $\omega$, $\wk_X$, and $\wck_X$ coincide on $\M(\Omega)$. Therefore the properties \q{V} and \qos{V} are equivalent for $C_0(\Omega)$.
\end{remark}

\begin{proof}
Throughout the proof we identify the dual of $C_0(\Omega)$ with the space $\M(\Omega)$ of all finite complex (or signed in the real case) Radon measures on $\Omega$.
By Proposition \ref{pel-prop-*-omega} it suffices to show that for every $K\subset (C_0(\Omega))^*=\M(\Omega)$ bounded $\omega(K)\leq \pi \eta(K)$. Let $K$ be a bounded subset of $\M(\Omega)$. From \cite[Proposition 5.2]{qrdpp} it follows that
$$\frac{1}{\pi}\omega(K)\leq \sup\left\{\limsup_{k\to\infty}\sup_{\mu\in K} |\mu(U_k)|:\ U_k\subset\Omega,\, k\in\N, \text{ pairwise disjoint, open} \right\}$$
(in the real case $\frac{1}{\pi}$ can be replaced by $\frac12$).

Let us fix an arbitrary $\delta<\omega(K)$.
Using the above inequality we find a sequence $(U_n)$ of pairwise disjoint open subsets of $\Omega$ and a sequence $(\mu_n)$ in $K$ such that $|\mu_n(U_n)|>\frac{\delta}{\pi}$.
For each $n\in\N$ we find a continuous function $f_n$ on $\Omega$ with a~compact support such that $\|f_n\|=1$, $f_n=0$ outside $U_n$, and 
\begin{equation}
\label{odhad-munfn}
|\mu_n(f_n)|=\left|\int_\Omega f_n \mathrm{d}\mu_n \right|>\frac{\delta}{\pi}.
\end{equation}
Then for every $\mu\in (C_0(\Omega))^*$ and $n\in\N$
$$\sum_{i=1}^n |\mu(f_i)| \leq \sum_{i=1}^n |\mu|(|f_i|) = |\mu|\left(\sum_{i=1}^n |f_i|\right) \leq |\mu|(1) = \|\mu\|,$$
hence $\sum f_n$ is a wuC series in $C_0(\Omega)$, and $\sup_{\mu\in B_{(C_0(\Omega))^*}} \sum_{i=1}^\infty |\mu(f_i)| \leq 1$.
By (\ref{odhad-munfn}) we have
$$\limsup_n \sup_{\mu\in K} \left|\int f_n \mathrm{d}\mu\right| \geq \limsup_n \left|\int f_n \mathrm{d}\mu_n \right| \geq \frac{\delta}{\pi}.$$
From this we conclude that $\eta(K) \geq \frac{\delta}{\pi}$, and since $\delta<\omega(K)$ has been chosen arbitrarily, $\omega(K)\leq \pi\eta(K)$. In the real case we obtain the similar inequality with $2$ instead of $\pi$.
\end{proof}

Recall that a Banach space $X$ is an \emph{$L^1$~predual}, if the dual space $X^*$ is isometrical to a space $L^1(\Omega,\Sigma,\mu)$ for some measure space $(\Omega,\Sigma,\mu)$. In 1973 Johnsson and Zippin proved that every real $L^1$~predual has the property (V) (see \cite[Corrollary (i)]{jz}). We prove a quantitative version of this theorem using results of their paper and the quantitative version of Pe\l czy\' nski's theorem.

\begin{theorem}
\label{q-pel-pred}
Let $X$ be a real $L^1$~predual. Then $X$ has the quantitative properties \q{V} and \qos{V}.
\end{theorem}
\begin{proof}
Let $Y$ be a Banach space and $T\colon X\to Y$ be an operator. We prove that $\gamma(T)\leq 4 \uc(T)$, that is, $X$ enjoys \q{V}. From this is follows that $\gamma(T^*)\leq 8 \uc(T)$ by the quantitative version of the Gantmacher theorem (\ref{q-gantmacher}). But the quantities $\gamma$ and $\omega$ are equivalent on $X^*$ -- by \cite[Theorem 7.5]{kks} and (\ref{equiv}) we obtain $\omega(T^*)\leq 16 \uc(T)$, which means that $X$ has \qos{V}.

Let us fix $\delta<\gamma(T)=\gamma\big(T(B_X)\big)$. By the definition of $\gamma$ we can find a sequence $(x_n)$ in $B_X$ for which $\gamma\big(\{Tx_n:\,n\in\N\}\big)>\delta$. The space $\overline\span\{x_n:\,n\in\N\}$ is a~closed separable subspace of the $L^1$~predual $X$, hence by \cite[\S~23, Lemma 1]{lacey} we can find a separable $L^1$~predual $Z$ such that $\overline\span\{x_n:\,n\in\N\}\subset Z \subset X$.

By \cite{jz}, $Z$ is a quotient of $C(\Delta)$, where $\Delta=\{0,1\}^{\N}$ is the Cantor space. Let $q\colon C(\Delta) \to Z$ be a quotient map, i.e. $q\left(U_{C(\Delta)}\right) = U_Z$. Then $T\circ q\colon C(\Delta)\to Y$ is a bounded linear operator, and
$$
\begin{aligned}
2\omega((T\circ q)^*) &\stackrel{(\ref{equiv}),(\ref{non-equiv})}{\geq} \gamma((T\circ q)^*) \stackrel{(\ref{q-gantmacher})}{\geq} \gamma(T\circ q) = \gamma\left(T\circ q\left(B_{C(\Delta)}\right)\right) = \gamma\left(T\left(q\left(U_{C(\Delta)}\right)\right)\right) \\
&\hskip 2.7 mm = \gamma\left(T\left(U_Z\right)\right) = \gamma\left(T\left(B_Z\right)\right) \geq \gamma\left(\{Tx_n:\,n\in\N\}\right) > \delta.
\end{aligned}$$
\begin{sloppypar}Since $\Delta$ is compact, Theorem \ref{q-pel} gives $\omega((T\circ q)^*)\leq 2 \uc(T\circ q)$, and we thus get $\uc(T\circ q) > \frac{\delta }{4}$. Hence we can find a wuC series $\sum f_n$ in $C(\Delta)$ with $\sup_{\mu\in B_{(C(\Delta))^*}} \sum |\mu(f_n)| \leq 1$ such that $\ca\left(\left(\sum_{i=1}^n T(q f_i)\right)_n \right) > \frac{\delta}{4}$.\end{sloppypar}

We set $z_n=q(f_n)$, $n\in\N$. Then $\sum z_n$ is a wuC series in $Z\subset X$ with
$$\sup_{x^*\in B_{X^*}} \sum_{n=1}^\infty |x^*(z_n)| = \sup_{x^*\in B_{X^*}} \sum_{n=1}^\infty |(x^*\circ q)(f_n)| \leq \sup_{\mu\in B_{(C(\Delta))^*}} \sum |\mu(f_n)| \leq 1.$$
Furthermore, $\ca\left(\left(\sum_{i=1}^n Tz_i\right)_n \right) > \frac{\delta}{4}$. Hence $\uc(T)>\frac{\delta}{4}$. This inequality holds for every $\delta<\gamma(T)$, therefore $\gamma(T)\leq 4\uc(T)$.
\end{proof}

\begin{proposition}
\label{C(Omega)}\
\begin{enumerate}[(i)]
\item If $\Omega$ is a scattered locally compact space, then $C_0(\Omega)$ has the properties \q{V}, \qos{V}, and also \qo{V}.
\item If $\Omega$ is an uncountable separable metrizable locally compact space, then $C_0(\Omega)$ has the properties \q{V} and \qos{V}, but it does not enjoy the property \qo{V}.
\end{enumerate}
\end{proposition}
\begin{proof}
Let $\Omega$ be a scattered locally compact space. The space $C_0(\Omega)$ has the properties \q{V} and \qos{V} by Theorem \ref{q-pel}. Let $Y$ be a Banach space and $T\colon C_0(\Omega) \to Y$ an operator. Since for $\Omega$ scattered $C_0(\Omega)^*$ is isometric to $\ell^1(\Omega)$, we have $\omega(T) \leq 2 \omega(T^*)$ by \cite[Theorem 8.2]{kks}. Combining it with Theorem \ref{q-pel} we obtain $\omega(T) \leq 2\pi\uc(T)$, that is, $C_0(\Omega)$ has the property \qo{V}.

The second statement is proved in Section \ref{properties}.
\end{proof}

The proposition above shows that \qo{V} differs from the other two quantifications, but we do not know whether there is any difference between the properties \q{V} and \qos{V}.

\begin{question}
Is there a Banach space which has the property \q{V} but not the property \qos{V}?
\end{question}

There is one even more interesting open question whether the Pelczynski's property (V) is automatically quantitative or not.

\begin{question}
Is there a Banach space which has Pe\l czy\' nski's property (V) but not the quantitative version \q{V}?
\end{question}

%%%%%%%%%%%%%%%%%%%%%%%%%%%%%%%%%%%%%%%%%%%%%%%%%%%%%%%%%%%%%%%%%%%%%%%%%%%%%%%%
%%%%%%%%%%%%%%%%%%%%%%%%%%%%%%%%%%%%%%%%%%%%%%%%%%%%%%%%%%%%%%%%%%%%%%%%%%%%%%%%
\section{Some other properties of Banach spaces, their quantification and relationship to the property (V)}
\label{other}
%%%%%%%%%%%%%%%%%%%%%%%%%%%%%%%%%%%%%%%%%%%%%%%%%%%%%%%%%%%%%%%%%%%%%%%%%%%%%%%%
%%%%%%%%%%%%%%%%%%%%%%%%%%%%%%%%%%%%%%%%%%%%%%%%%%%%%%%%%%%%%%%%%%%%%%%%%%%%%%%%

In this section we remind the definitions of some known properties of operators between Banach spaces, relationships between them, and their relation to unconditionally converging operators. These relationships are then quantified. The introduced properties of operators give rise to some properties of Banach spaces which are related to Pe\l czy\' nski's property (V). These properties can be quantified in the same way as the property (V). Using the proved quantitative relationships between different kinds of operators we establish the relation between quantitative versions of relevant properties of Banach spaces, including the property (V). Finally, we apply these results and those of \cite{kks} to some $C_0(\Omega)$ spaces.

%%%%%%%%%%%%%%%%%%%%%%%%%%%%%%%%%%%%%%%%%%%%%%%%%%%%%%%%%%%%%%%%%%%%%%%%%%%%%%%%
\subsection{Some properties of operators, their relation to unconditionally converging operators, and their quantification}
%%%%%%%%%%%%%%%%%%%%%%%%%%%%%%%%%%%%%%%%%%%%%%%%%%%%%%%%%%%%%%%%%%%%%%%%%%%%%%%%

Let $X$ be a Banach space. We will denote by $\rho$ the topology of uniform convergence on weakly compact subsets of $X^*$. This topology is called the \emph{Right topology}\/ and it is the restriction to $X$ of the Mackey topology $\mu(X^{**},X^*)$ on $X^{**}$ with respect to the dual pair $(X^{**},X^*)$. An operator from $X$ into a Banach space $Y$ is weakly compact if and only if it is Right-to-norm continuous (see \cite{pvwy}).

We say that an operator between Banach spaces is
\begin{itemize}
\item \emph{completely continuous (cc)} if it is weak-to-norm sequentially continuous,
\item \emph{pseudo weakly compact (pwc)} if it is Right-to-norm sequentially continuous,
\item \emph{weakly completely continuous (wcc)} if it maps weakly Cauchy sequences to weakly convergent sequences,
\item \emph{Right completely continuous (Rcc)} if it maps Right-Cauchy sequences to Right-convergent sequences.
\end{itemize}
M. Ka\v cena has proved in \cite[\S~3]{kacena} (using also \cite{pvwy}) that for every operator $T$ between Banach spaces the following implications hold:
\begin{equation}
\label{pavouk}
\begin{array}{ccccccccc}
&& \hskip -2 mm T \text{ is $w$-compact} \hskip -2 mm & \Longrightarrow & \hskip -8 mm T \text{ is pwc} \hskip -5 mm &&&& \cr
T \text{ is compact} \hskip -2 mm & \text{\rotatebox[origin=l]{45}{$\hskip 0.5 mm \Longrightarrow$}} \hskip -5,6 mm \text{\rotatebox[origin=l]{-45}{$\hskip 0.5 mm\Longrightarrow$}} && \text{\rotatebox[origin=c]{45}{$\Longrightarrow$}} \hskip -5,2 mm \text{\rotatebox[origin=c]{-45}{$\Longrightarrow$}}&&\text{\rotatebox[origin=r]{45}{$\hskip -4 mm \Longrightarrow$}} \hskip -2,4 mm \text{\rotatebox[origin=r]{-45}{$\hskip -4 mm \Longrightarrow$}}&\hskip -2 mm T \text{ is Rcc} \hskip -2 mm & \Longrightarrow & \hskip -2 mm T \text{ is uc} \cr
&&T \text{ is cc} & \Longrightarrow & \hskip -2 mm T \text{ is wcc\quad} \hskip -2 mm  &&&& \cr
\end{array}
\end{equation}
Some of these implications have already been quantified in \cite[\S~3,4]{kks}. In this section we quantify the rest.

Let $X$, $Y$ be Banach spaces and $T\colon X\to Y$ an operator. We set
$$
\begin{aligned}
\cc(T) &= \sup\big\{\ca\big((Tx_n)\big):\ (x_n) \text{ is a weakly Cauchy sequence in } B_X\big\}, \\
\cc_\rho(T) &= \sup\big\{\ca\big((Tx_n)\big):\ (x_n) \text{ is a Right-Cauchy sequence in } B_X\big\}. \\
\end{aligned}
$$
The former quantity measures how far is $T$ from being completely continuous, the latter one measures how far is $T$ from being pseudo weakly compact.

As for the other properties mentioned above, let us first remind that a bidual space $X^{**}$ is complete with respect to both the weak$^*$ and the Mackey topology and that the weak$^*$ topology is coarser than the Mackey topology. Therefore every weakly Cauchy sequence in a Banach space $X$ is weak$^*$-convergent in $X^{**}$, every Right-Cauchy sequence in a Banach space $X$ is $\mu(X^{**},X^*)$-convergent and hence also weak$^*$-convergent in $X^{**}$. Each bounded linear operator, which is by definition norm-to-norm continuous, is also weak-to-weak continuous and Right-to-Right continuous (see \cite[Lemma 12]{pvwy}). Let $X$, $Y$ be Banach spaces and $T\colon X\to Y$ an operator. Let us set
$$\begin{aligned}
\wcc(T) &= \sup\big\{\dist(w^*\text{-}\lim(Tx_n), Y):~(x_n) \text{ is a } w\text{-Cauchy sequence in } B_X\big\} \\
        &= \sup\big\{\wk_Y\big(\{Tx_n:\,n\in\N\}\big):~(x_n) \text{ is a } w\text{-Cauchy sequence in } B_X\big\}, \\
\wcc_\omega(T) &= \sup\big\{\omega\big(\{Tx_n:\,n\in\N\}\big):~(x_n) \text{ is a } w\text{-Cauchy sequence in } B_X\big\}, \\
\Rcc(T) &= \sup\big\{\dist(\mu(Y^{**}\hskip -1 mm,Y^*)\text{-}\lim(Tx_n), Y):\ (x_n) \text{ is a } \rho\text{-Cauchy sequence in } B_X\big\} \\
        &= \sup\big\{\dist(w^*\text{-}\lim(Tx_n), Y):~(x_n) \text{ is a } \rho\text{-Cauchy sequence in } B_X\big\} \\
				&= \sup\big\{\wk_Y\big(\{Tx_n:\,n\in\N\}\big):~(x_n) \text{ is a } \rho\text{-Cauchy sequence in } B_X\big\}, \\
\Rcc_\omega(T) &= \sup\big\{\omega\big(\{Tx_n:\,n\in\N\}\big):~(x_n) \text{ is a } \rho\text{-Cauchy sequence in } B_X\big\}. \\
\end{aligned}
$$
The first two quantities measure (in two different ways) weak non-complete continuity of $T$, the last two are measures of Right non-complete continuity of $T$.

The following theorem contains quantitative versions of all the implications in (\ref{pavouk}).

\begin{theorem}
\label{pavouk-quant}
Let $X$, $Y$ be Banach spaces and $T\colon X\to Y$ an operator. Then
$$
\begin{array}{ccccclc}
&& 2\omega(T^*) &&&& \cr
&& \text{\rotatebox[origin=c]{90}{$\leq$}} &&&& \cr
&& \cc_\rho(T) & \leq & \cc(T) & \leq 4\cdot \hskip -3 mm & \chi(T) \cr
&& \text{\rotatebox[origin=c]{90}{$\leq$}} && \text{\rotatebox[origin=c]{90}{$\leq$}} && \text{\rotatebox[origin=c]{90}{$\leq$}} \cr
&&\Rcc_\omega(T) & \leq & \wcc_\omega(T) & \leq & \omega (T) \cr
&& \text{\rotatebox[origin=c]{90}{$\leq$}} && \text{\rotatebox[origin=c]{90}{$\leq$}} && \text{\rotatebox[origin=c]{90}{$\leq$}} \cr
\frac14 \uc(T) & \leq & \Rcc(T) & \leq & \wcc(T) & \leq & \wk_Y(T). \cr
\end{array}
$$
\end{theorem}

\begin{proof}
All the inequalities
$$
\begin{array}{ccc}
4\chi(T) && \cr
\text{\rotatebox[origin=c]{90}{$\leq$}} && \cr
\cc(T) & \leq & \chi(T)\cr
\text{\rotatebox[origin=c]{90}{$\leq$}} && \text{\rotatebox[origin=c]{90}{$\leq$}} \cr
\wcc_\omega(T) & \leq & \omega (T) \cr
\text{\rotatebox[origin=c]{90}{$\leq$}} && \text{\rotatebox[origin=c]{90}{$\leq$}} \cr
\wcc(T) & \leq & \wk_Y(T) \cr
\end{array}
$$
has already been proved (or simply observed) in \cite[\S~3]{kks}. The inequality $\cc_\rho(T) \leq 2 \omega(T^*)$ follows from \cite[(2.1) and (4.1)]{kks}.

The inequalities $\cc_\rho(T) \leq \cc(T)$, $\Rcc(T) \leq \wcc(T)$, and $\Rcc_\omega(T) \leq \wcc_\omega(T)$ are trivial, since every Right-Cauchy sequence is weakly Cauchy. By (\ref{non-equiv}), $\wk_Y(A)\leq \omega(A)$ for every bounded $A\subset Y$. Therefore $\Rcc(T)\leq \Rcc_\omega(T)$ (as well as $\wcc(T)\leq \wcc_\omega(T)$, which has already been noted).

Let us show that $\Rcc_\omega(T) \leq \cc_\rho(T)$. Suppose that $\cc_\rho(T)<\delta$. Let $(x_n)$ be a~Right-Cauchy sequence in $B_X$. Then $\ca\big((Tx_n)\big)<\delta$, hence we can find $n_0\in\N$ such that $\|Tx_n-Tx_{n_0}\|<\delta$ whenever $n>n_0$. Set $K=\{Tx_1,\dots,Tx_{n_0}\}$. Then $K$ is weakly compact, and $\hd(\{Tx_n:\,n\in\N\},K) \leq \delta$. Therefore $\omega(\{Tx_n:\,n\in\N\}) \leq \delta$. We thus get $\Rcc_\omega(T) \leq \delta$, and consequently $\Rcc_\omega(T) \leq \cc_\rho(T)$.

Finally, we prove the inequality $\uc(T) \leq 4\Rcc(T)$. By Theorem \ref{fixc0-quant}, it is enough to show that $\fc(T) \leq 2\Rcc(T)$. If $\fc(T)=0$, then it is obvious. Suppose that $\fc(T)>0$ and fix $0<\delta<\fc(T)$. By the definition of $\fc(T)$ we find a subspace $X_0$ of $X$ isomorphic to $c_0$ and onto isomorphisms $U\colon c_0\to X_0$, $V\colon T(X_0)\to c_0$ such that $(T\restriction_{X_0})^{-1}=U\circ V$, and $\big(\|U\| \|V\|\big)^{-1} > \delta$.

Set $f_n=\sum_{i=1}^n e_n \in c_0$, $n\in\N$. Then $(f_n)$ is a weakly Cauchy sequence in $c_0$. Since the space $c_0$ enjoys the Dunford-Pettis property (see e.g. \cite[p. 597]{bst}), the weak and the Right topology coincide sequentially on it by \cite[Proposition 3.17]{kacena}. Therefore the sequence $(f_n)$ is Right-Cauchy. Let us define $x_n=\frac{1}{\|U\|}Uf_n$, $n\in\N$. By the continuity of $U$, $(x_n)$ is a Right-Cauchy sequence in $B_X$. Since $T$ is bounded, we also have that $(Tx_n)$ is a Right-Cauchy sequence in $Y$. Let $y^{**}$ be its $\mu(Y^{**},Y^*)$-limit in $Y^{**}$. We will show that $\dist(y^{**},Y)>\frac{\delta}{2}$.

Let us set $Y_0 = \overline{\span}\{Tx_n:\,n\in\N\} = \overline{\span}\{(T\circ U)f_n:\,n\in\N\} = \overline{\span}\{(T\circ U)e_n:\,n\in\N\} = (T\circ U)(c_0)$. If $T\circ U$ is regarded as an isomorphism from $c_0$ onto $Y_0$, then $(T\circ U)^{**}$ is an isomorphism from $c_0^{**}$ onto $Y_0^{**}$, and $\|((T\circ U)^{**})^{-1}\| = \|(T\circ U)^{-1}\| = \|V\|$.
Let $y_0\in Y_0$ be arbitrary. We find $z\in c_0$ which satisfies $\tfrac{1}{\|U\|}(T\circ U)z=y_0$. Then
$$
\begin{aligned}
\|y^{**}-y_0\| &= \left\|\mu(Y^{**},Y^*)\text{-}\lim\Big(\tfrac{1}{\|U\|}(T\circ U)f_n\Big) - \tfrac{1}{\|U\|}(T\circ U)z\right\| \\
&= \frac{1}{\|U\|} \left\|(T\circ U)^{**}\big((\mu(Y^{**},Y^*)\text{-}\lim f_n) - z \big)\right\| \\
&\geq\frac{1}{\|U\|} \left\|\big((T\circ U)^{**}\big)^{-1}\right\|^{-1} \left\|(\mu(Y^{**},Y^*)\text{-}\lim f_n) - z\right\| \\
&= \|U\|^{-1} \|V\|^{-1} \left\|(w^*\text{-}\lim f_n) - z\right\| \geq (\|U\| \|V\|)^{-1},
\end{aligned}
$$
where the last inequality follows from the fact, that $w^*\text{-}\lim f_n = (1,1,1,\dots)\in \ell^\infty\cong c_0^{**}$ whereas $z\in c_0$, so the distance between these two elements is at least $\lim_{n\to\infty}|1-z(n)| = 1$. Therefore $\dist(y^{**},Y_0)\geq (\|U\| \|V\|)^{-1} > \delta$. By \cite[Lemma 2.2]{gs}, $\dist(y^{**},Y_0) \leq 2\dist(y^{**},Y)$, and hence $\dist(y^{**},Y) > \frac{\delta}{2}$.

We thus have $\Rcc(T) > \frac{\delta}{2}$. It follows that $\fc(T)\leq 2 \Rcc(T)$, which completes the proof.
\end{proof}

\begin{remark}
The inequality $\cc_\rho(T) \leq 2 \omega(T^*)$ from the above theorem quantifies the implication
$$T \text{ is weakly compact } \Rightarrow T \text{ is pseudo weakly compact}$$
due to the Gantmacher theorem. We cannot obtain a better quantification either with $\gamma(T)$ or with $\omega(T)$ instead of $\omega(T^*)$. The space $X$ constructed in \cite[Example 10.1(v)]{kks} forms a counterexample. Since this space enjoys the Dunford-Pettis property, the weak and the Right topology coincide sequentially on $X$ (see \cite[Proposition 3.17]{kacena}), thus $\cc(T) = \cc_\rho(T)$ for each operator $T\colon X\to Y$ ($Y$ a Banach space). But there are operators $T_n\colon X\to c_0$, $n\in\N$, such that $\cc(T_n)\geq 1$ for each $n\in\N$ and $\omega(T_n)=\wk_{c_0}(T_n)\to 0$. The measures $\wk_{c_0}$ and $\gamma$ are equivalent by (\ref{equiv}), hence there is no constant $C>0$ such that $\cc_\rho(T)=\cc(T) \leq C \gamma(T)$ or $\cc_\rho(T)=\cc(T) \leq C \omega(T)$ for each operator $T\colon X\to c_0$. 
\end{remark}

%%%%%%%%%%%%%%%%%%%%%%%%%%%%%%%%%%%%%%%%%%%%%%%%%%%%%%%%%%%%%%%%%%%%%%%%%%%%%%%%
\subsection{Properties of Banach spaces related to above-defined properties of operators and a relationship between their quantitative versions}
\label{properties}
%%%%%%%%%%%%%%%%%%%%%%%%%%%%%%%%%%%%%%%%%%%%%%%%%%%%%%%%%%%%%%%%%%%%%%%%%%%%%%%%

Let us recall some properties of Banach spaces, whose definitions use the above-introduced properties of operators. We follow the the notation of \cite{kacena}.
Let $X$ be a Banach space. We say that
\begin{itemize}
\item $X$ has the \emph{reciprocal Dunford-Pettis property (RDP)}\/ if for every Banach space $Y$ every cc operator $T\colon X\to Y$ is weakly compact,
\item $X$ has the \emph{Dieudonn\'e property (D)}\/ if for every Banach space $Y$ every wcc operator $T\colon X\to Y$ is weakly compact,
\item $X$ has the \emph{Right Dieudonn\'e property (RD)}\/ if for every Banach space $Y$ every Rcc operator $T\colon X\to Y$ is weakly compact,
\item $X$ is \emph{sequentially Right (SR)}\/ if for every Banach space $Y$ every pwc operator $T\colon X\to Y$ is weakly compact.
\end{itemize}

The following implications are an immediate consequence of (\ref{pavouk}):
\begin{equation}
\label{pavouk-p}
\begin{array}{ccccccc}
&&&& X \text{ is (SR)} && \cr
X \text{ has (V)} & \Longrightarrow & X \text{ has (RD)} & \text{\rotatebox[origin=l]{45}{$\hskip 0.5 mm \Longrightarrow$}} \hskip -5,6 mm \text{\rotatebox[origin=l]{-45}{$\hskip 0.5 mm\Longrightarrow$}}&&\text{\rotatebox[origin=r]{45}{$\hskip -4 mm \Longrightarrow$}} \hskip -2,4 mm \text{\rotatebox[origin=r]{-45}{$\hskip -4 mm \Longrightarrow$}}& X \text{ has (RDP)} \cr
&&&& X \text{ has (D)} && \cr
\end{array}
\end{equation}
All these properties have their quantitative versions, obtained in a standard way. First we define quantitative versions of the properties (SR) and (RDP) analogous to \q{V}, \qo{V}, \qos{V}.
\begin{definition}
We say that a Banach space $X$ has the property \q{RDP}, \qo{RDP}, or \qos{RDP} if there is a constant $C>0$ such that for every Banach space $Y$ and every operator $T\colon X\to Y$
$$\gamma(T)\leq C \cc(T), \qquad \omega(T)\leq C \cc(T), \qquad \text{or} \qquad \omega(T^*) \leq C \cc(T),$$
respectively.
Analogously we define the properties \q{SR}, \qo{SR}, and \qos{SR} -- we just replace $\cc$ in the above inequalities by $\cc_\rho$.
\end{definition}
For the details about a quantification of the reciprocal Dunford-Pettis property we refer the reader to \cite{qrdpp}.
Regarding the properties (D) and (RD), there are even more possibilities of quantification. Besides the measures of weak non-compactness of $T$, we can also choose between two different quantities which measure weak non-complete continuity and Right non-complete continuity of $T$.
\begin{definition}
We say that a Banach space $X$ has the property \q{D}, \qo{D}, \qos{D}, \qO{D}, \qoO{D}, or \qosO{D} if there is a constant $C>0$ such that for every Banach space $Y$ and every operator $T\colon X\to Y$
$$
\def\arraystretch{1.5}
\begin{array}{lll}
\gamma(T)\leq C \wcc(T), \hskip 1 cm &\omega(T)\leq C \wcc(T),\hskip 1 cm &\omega(T^*) \leq C \wcc(T), \\
\gamma(T)\leq C \wcc_\omega(T),  &\omega(T)\leq C \wcc_\omega(T),\quad \text{or} &\omega(T^*) \leq C \wcc_\omega(T).
\end{array}
$$
The properties \q{RD}, \qo{RD}, \qos{RD}, \qO{RD}, \qoO{RD}, and \qosO{RD} are defined in the same way, the quantities $\wcc$ and $\wcc_\omega$ are replaced by $\Rcc$ and $\Rcc_\omega$, respectively.
\end{definition}

Clearly, if $X$ has the property \qo{P} or \qos{P}, then it also has the property \q{P} by (\ref{equiv}), (\ref{non-equiv}), and (\ref{q-gantmacher}). Here P stands for V, RD, D, SR, or RDP. From Theorem \ref{pavouk-quant} we obtain the following relationships between the quantitative versions of the properties defined above.

\begin{theorem}
\label{pavouk-p-quant}
For a Banach space $X$ the following implications hold:
\[
\begin{array}{ccccccc}
X \text{ has \q{V}} & \Longrightarrow & X \text{ has \q{RD}} & \Longrightarrow & X \text{ has \qO{RD}} & \Longrightarrow & X \text{ has \q{SR}} \cr
&& \text{\rotatebox[origin=c]{-90}{$\Longrightarrow$}} && \text{\rotatebox[origin=c]{-90}{$\Longrightarrow$}} && \text{\rotatebox[origin=c]{-90}{$\Longrightarrow$}} \cr
&& X \text{ has \q{D}} & \Longrightarrow & X \text{ has \qO{D}} & \Longrightarrow & X \text{ has \q{RDP}}, \cr
\end{array}
\]
\medskip
\[
\begin{array}{ccccccc}
X \text{ has \qo{V}} & \Longrightarrow & X \text{ has \qo{RD}} & \Longrightarrow & X \text{ has \qoO{RD}} & \Longrightarrow & X \text{ has \qo{SR}} \cr
&& \text{\rotatebox[origin=c]{-90}{$\Longrightarrow$}} && \text{\rotatebox[origin=c]{-90}{$\Longrightarrow$}} && \text{\rotatebox[origin=c]{-90}{$\Longrightarrow$}} \cr
&& X \text{ has \qo{D}} & \Longrightarrow & X \text{ has \qoO{D}} & \Longrightarrow & X \text{ has \qo{RDP}}, \cr
\end{array}
\]
\medskip
\[
\begin{array}{ccccccc}
X \text{ has \qos{V}} & \Longrightarrow & X \text{ has \qos{RD}} & \Longrightarrow & X \text{ has \qosO{RD}} & \Longrightarrow & X \text{ has \qos{SR}} \cr
&& \text{\rotatebox[origin=c]{-90}{$\Longrightarrow$}} && \text{\rotatebox[origin=c]{-90}{$\Longrightarrow$}} && \text{\rotatebox[origin=c]{-90}{$\Longrightarrow$}} \cr
&& X \text{ has \qos{D}} & \Longrightarrow & X \text{ has \qosO{D}} & \Longrightarrow & X \text{ has \qos{RDP}}. \cr
\end{array}
\]
\end{theorem}

Now we are ready to prove Proposition \ref{C(Omega)}\,(ii).

\begin{proof}[Proof of Proposition \ref{C(Omega)}\,(ii)]
If we take the space $Y$ constructed in \cite[Example 3.2]{qrdpp}, then there is a sequence $(T_n)$ of operators from $C_0(\Omega)$ to $Y$ which satisfies
$$\lim_{n\to\infty} \frac{\cc(T_n)}{\omega(T_n)}=0.$$
Therefore $C_0(\Omega)$ does not have the property \qo{RDP}, hence not even the property \qo{V} by Theorem \ref{pavouk-p-quant}. But it follows from Theorem \ref{q-pel} that $C_0(\Omega)$ enjoys the property \q{V}.
\end{proof}

%%%%%%%%%%%%%%%%%%%%%%%%%%%%%%%%%%%%%%%%%%%%%%%%%%%%%%%%%%%%%%%%%%%%%%%%%%%%%%%%
\subsection{Some corollaries for $C_0(\Omega)$ spaces}
\label{corollaries}
%%%%%%%%%%%%%%%%%%%%%%%%%%%%%%%%%%%%%%%%%%%%%%%%%%%%%%%%%%%%%%%%%%%%%%%%%%%%%%%%

\begin{corollary}
\label{equiv-C0}
Let $\Omega$ be a locally compact space, $Y$ be a Banach space, and $T\colon C_0(\Omega) \to Y$ an operator. Then
$$
\begin{array}{ccccccccccccc}
&& \wcc(T) && \wcc_\omega(T) && \cc(T) &&&& \cr
&& \text{\rotatebox[origin=c]{90}{$=$}} && \text{\rotatebox[origin=c]{90}{$=$}} && \text{\rotatebox[origin=c]{90}{$=$}} &&&& \cr
\frac14 \uc(T) & \leq & \Rcc(T) & \leq & \Rcc_\omega(T) & \leq & \cc_\rho(T) & \leq  & 2\omega(T^*) & \leq & 2\pi \uc(T) \cr
&& \text{\rotatebox[origin=c]{-90}{$\leq$}} &&&&&& \text{\rotatebox[origin=c]{-90}{$=$}} && \cr
&& \wk_Y(T) & \leq & \gamma(T) & \leq & \gamma(T^*) & \leq & 2 \wk_Y(T^*) && \cr
&& \text{\rotatebox[origin=c]{-90}{$\leq$}} &&&&&&&& \cr
&& \omega(T) &&&&&&&& \cr
&& \text{\rotatebox[origin=c]{-90}{$\leq$}} &&&&&&&& \cr
&& \chi(T). &&&&&&&& \cr
\end{array}
$$
In particular, all the quantities except for $\omega(T)$ and $\chi(T)$ are equivalent.
\end{corollary}
\begin{proof}
Since $C_0(\Omega)$ has the Dunford-Pettis property, the weak and the Right topology coincide sequentially on $X$ by \cite[Proposition 3.17]{kacena}. That is why $\Rcc(T) = \wcc(T)$, $\Rcc_\omega(T) = \wcc_\omega(T)$, and $\cc_\rho(T) = \cc(T)$. The equality $\omega(T^*) = \wk_Y(T^*)$ follows from \cite[Theorem 7.5]{kks}. By (\ref{equiv}) and (\ref{q-gantmacher}) we have $\wk_Y(T) \leq \gamma(T) \leq \gamma(T^*) \leq 2 \wk_Y(T^*)$. Theorem \ref{q-pel} gives $\omega(T^*) \leq \pi \uc(T)$. The rest follows from Theorem \ref{pavouk-quant}.
\end{proof}

\begin{remark}
Almost the same assertion holds for every operator $T\colon X\to Y$ if $X$ is a real $L^1$ predual and $Y$ a Banach space. We only need to adjust the constant in the inequality $\omega(T^*) \leq \pi \uc(T)$. From the proof of Theorem \ref{q-pel-pred} we see that it is enough to replace $\pi$ by $16$. All the quantities except for $\omega(T)$ and $\chi(T)$ are still equivalent.
\end{remark}

\begin{corollary}
\label{equiv-scattered}
Let $\Omega$ be a scattered locally compact space, $Y$ be a Banach space, and $T\colon C_0(\Omega) \to Y$ an operator. Then
$$
\begin{array}{ccccccccccccc}
&& \wcc(T) && \wcc_\omega(T) && \cc(T) &&&& \cr
&& \text{\rotatebox[origin=c]{90}{$=$}} && \text{\rotatebox[origin=c]{90}{$=$}} && \text{\rotatebox[origin=c]{90}{$=$}} &&&& \cr
\frac14 \uc(T) & \leq & \Rcc(T) & \leq & \Rcc_\omega(T) & \leq & \cc_\rho(T) & \leq  & 2\omega(T^*) & \leq & 2\pi \uc(T) \cr
&& \text{\rotatebox[origin=c]{-90}{$\leq$}} &&&&\text{\rotatebox[origin=c]{90}{$\leq$}} &&&& \cr
&& \wk_Y(T) & \leq & \omega(T) & \leq & \chi(T) &&&& \cr
\end{array}
$$
Hence all the involved quantities are equivalent.
\end{corollary}
\begin{proof}
The assertion follows from Corollary \ref{equiv-C0} and \cite[Theorem 8.2]{kks} since $C_0(\Omega)^*$ for $\Omega$ scattered is isometric to $\ell^1(\Omega)$.
\end{proof}

%%%%%%%%%%%%%%%%%%%%%%%%%%%%%%%%%%%%%%%%%%%%%%%%%%%%%%%%%%%%%%%%%%%%%%%%%%%%%%%%
%%%%%%%%%%%%%%%%%%%%%%%%%%%%%%%%%%%%%%%%%%%%%%%%%%%%%%%%%%%%%%%%%%%%%%%%%%%%%%%%
\section*{Acknowledgement} The author wishes to thank her supervisor Ond\v rej Kalenda for many helpful conversations, suggestions, and comments.
%%%%%%%%%%%%%%%%%%%%%%%%%%%%%%%%%%%%%%%%%%%%%%%%%%%%%%%%%%%%%%%%%%%%%%%%%%%%%%%%
%%%%%%%%%%%%%%%%%%%%%%%%%%%%%%%%%%%%%%%%%%%%%%%%%%%%%%%%%%%%%%%%%%%%%%%%%%%%%%%%

\bibliography{qpelczynski}
\bibliographystyle{plain}

\end{document}